\documentclass{article}
\usepackage{amsthm}
\usepackage{amsmath}
\usepackage{amsfonts,amssymb,amscd}
\usepackage{mathrsfs,mathtools}
\usepackage{hyperref,latexsym}
\usepackage[figurename=Fig.]{caption}
\usepackage{indentfirst}
\usepackage{bm}
\usepackage{booktabs}
\usepackage[table,xcdraw]{xcolor}
\usepackage{multirow}
\usepackage{rotating,tabularx}
\usepackage{subfig}
\usepackage{tkz-euclide}
\usepackage{tkz-euclide}
\usepackage{tkz-berge}
\usepackage{xparse}
\usepackage{amsmath,amssymb}
\usetikzlibrary{topaths,calc,decorations.pathreplacing,hobby,fit,scopes,matrix,positioning,decorations.pathmorphing,automata,backgrounds}

\tikzset{
 dot/.style={draw, circle, inner sep=0.5pt, fill=blue}
}

\pgfkeys{
    /rootG/.is family, /rootG,
    default/.style = {rotate=0, scale=0.8},
    rotate/.estore in = \rootGRotate,
    scale/.estore in = \rootGScale,
}

\DeclareDocumentCommand\rootG{O{}m}{
    \pgfkeys{/rootG, default, #1}
    \begin{scope}[shift={#2},rotate=\rootGRotate,scale=\rootGScale,thick]
        \draw[fill=cyan!10!white,fill opacity=0.5] (0,0) .. controls (-1.2,1.5) and (-1.2,2.5) .. (0,2.5) 
            .. controls (1.2,2.5) and (1.2,1.5) .. (0,0);
        \node at (0,1.5) {$G_r$};
    \end{scope}
}

\DeclareDocumentCommand\longpath{m m m g g}{
\def \n {#3}  \def \startn {#1}  \def \endn {#2}
\IfNoValueTF{#5}{\def \p {p}}{\def \p {#5} }
\foreach \v in {0,...,\n} {\coordinate (\p\v) at ($(\startn)!\v/\n!(\endn)$);}
\IfNoValueTF{#4}
{\foreach \x [evaluate=\x as \y using int(\x-1)] in {1,...,\n}{\draw[black](\p\y) node {}--(\p\x) node {};}}
{\foreach \x [evaluate=\x as \y using int(\x-1)] in {1,...,\n}{
\IfValueInList \x {#4}
{\foreach \v in {4,5,6}{\fill ($ (\p\y)!\v/10! (\p\x) $) circle (0.5pt);}}
{\draw[black](\p\y) node {}--(\p\x) node {};}
}}}

\DeclareDocumentCommand\nstar{o m m m m g g}{
\def \o {#2}  \def \n {#3}  \def \deg {#4}  \def \len {#5}
\IfNoValueTF{#1}{\def \rot {0}}{\def \rot {#1} }
\IfNoValueTF{#7}{\def \p {p}}{\def \p {#7} }
\foreach \x in {0,1,...,\n} {\coordinate[rotate around={-0.5*\deg+\rot: (\o)}](\p\x) at ($(\o)+({\deg * \x/\n}: \len)$);}
\foreach \x[evaluate=\x as \y using int(\x+1),evaluate=\x as \z using int(\x-1)] in {0,1,...,\n}{
  \IfValueInList \x {#6}{\foreach \v in {2,4,6}{
  \fill[black] ($ (\p\z)!\v/8! (\p\y) $) circle (0.5pt);
   }}
  {\draw (\o) node {}-- (\p\x)node{};}
}
}

\DeclareDocumentCommand\star{o m m m m g g}{
\def \o {#2}  \def \n {#3}  \def \deg {#4}  \def \len {#5}
\IfNoValueTF{#1}{\def \rot {0}}{\def \rot {#1} }
\IfNoValueTF{#7}{\def \p {p}}{\def \p {#7} }
\foreach \x in {0,1,...,\n} {\coordinate[rotate around={-0.5*\deg+\rot: (\o)}](\p\x) at ($(\o)+({\deg * \x/\n}: \len)$);}
\foreach \x[evaluate=\x as \y using int(\x+1),evaluate=\x as \z using int(\x-1)] in {0,1,...,\n}{

  {\draw (\o) node {}-- (\p\x)node{};}
}
}

\DeclareDocumentCommand\fitellipsis{o m m m m}{
\IfNoValueTF{#1}{\def \fc {gray}}{\def \fc {#1}}
\draw[fill=\fc!25, fill opacity=0.5] let \p1=(#2), \p2=(#3), \n1={atan2(\y2-\y1,\x2-\x1)}, \n2={veclen(\y2-\y1,\x2-\x1)} in ($ (\p1)!0.5!(\p2) $) ellipse [radius=\n2/2+#4pt, y radius=#5pt, rotate=\n1];
\foreach \v in {4,5,6} {
\fill ($ (#2)!\v/10!(#3) $) circle (0.5pt);
}
}

\newcommand{\dotellipsis}[4] 
{\draw[dashed] let \p1=(#1), \p2=(#2), \n1={atan2(\y2-\y1,\x2-\x1)}, \n2={veclen(\y2-\y1,\x2-\x1)}
    in ($ (\p1)!0.5!(\p2) $) ellipse [radius=\n2/2+#3pt, y radius=#4pt, rotate=\n1];
}

\DeclareDocumentCommand\hyperpath{o m m m }{
\IfNoValueTF{#1}{\def \fco {gray}}{\def \fco {#1}}
\pgfmathtruncatemacro{\len}{#4}
\foreach \v in {0,...,\len} {
	\coordinate (p\v) at ($(#2)!\v/\len!(#3)$);
	 }

\pgfmathtruncatemacro{\lena}{\len-1}
\foreach \cur  [count=\next from 1] in {0,...,\lena} {
	\fitellipsis[\fco]{p\cur}{p\next}{8}{6};
}
\foreach \v in {0,...,\len} {
   \node[shade,shading=ball,circle,ball color=red!90!white,minimum size=1.8pt,inner sep=0pt] at  (p\v) {};
	 }
}

\newcommand{\partellipsis}[4]{
\tkzDefMidPoint(#1,#2) \tkzGetPoint{tikzC}
\tkzCalcLength[cm](#1,tikzC)\tkzGetLength{dAC}
\tkzFindSlopeAngle(#1,#2)\tkzGetAngle{tkzang}
\draw[rotate around={\tkzang: (tikzC)}] ($(tikzC) + ({(\dAC cm + #3 pt)* cos(180-70)}, { #4 pt * sin(180-70)})$) arc [start angle=180-70, end angle=180+70, x radius=\dAC cm+ #3 pt, y radius= #4 pt];
\draw[rotate around={\tkzang: (tikzC)}] ($(tikzC) + ({(\dAC cm + #3 pt) * cos(-70)}, { #4 pt * sin(-70)})$) arc [start angle=-70, end angle=70, x radius=\dAC cm+ #3 pt, y radius= #4 pt];
\node[rotate around={\tkzang: (tikzC)}] at (tikzC) {$\cdots$};
}
\usepackage[backend=biber,style=numeric,doi=false,isbn=false,url=false,hyperref=true,backref=true]{biblatex}
\usepackage{adjustbox}
\usepackage{indentfirst}
\usepackage{enumerate,longtable, tabu}
\graphicspath{{./img/}{./figures}}
\usepackage{tikz}

\usetikzlibrary{calc}

\newtheorem{thm}{Theorem}[section]
\newtheorem{lem}{Lemma}[section]

\newtheorem{proof*}{proof}[section]
\newtheorem{cor}{Corollary}[section]
\newtheorem{con}{Conjecture}[section]

\newtheorem{question}{Question}

\newtheorem{fact}{Fact}

\newcommand{\x}{{\ensuremath{ \boldsymbol{x}}}}
\newcommand{\y}{{\ensuremath{ \mathbf{y}}}}
\newcommand{\z}{{\ensuremath{ \mathbf{z}}}}

\DeclareMathOperator{\sgn}{sgn}

\title{The distance spectral radius of $k$-uniform hypertrees with given number of vertices of maximum degree.
\footnote{\noindent This work is    partially supported by 
			the National Natural Science Foundation of China (No.~12271182).\\ 
   E-mail: 2111152@tongji.edu.cn (X. Liu),shan\_haiying@tongji.edu.cn (H. Shan)}}

\author{ Xiaoqi, Liu,\, Haiying Shan\thanks{Corresponding author}
\\ \small School of Mathematical Sciences, Key Laboratory of Intelligent Computing and Applications\\ \small(Ministry of Education), Tongji University, Shanghai, China}

\date{}

\begin{document}
\maketitle
	\begin{abstract}
  This paper investigates the influence of two graft transformations on the distance spectral radius of connected uniform hypergraphs. Specifically, we study $k$-uniform hypertrees with given size, maximum degree and number of vertices of maximum degree, and give the structure of such hypergraph with maximum distance spectral radius.
		
	\end{abstract}
\section{Introduction}
A hypergraph $G$ consists of a nonempty vertex set $V(G)$ and an edge set $E(G)$, where each edge $e\in E(G)$ is a nonempty subset of $V(G)$. The order of $G$ is denoted as $|V(G)|$, and the size of $G$ is denoted as $|E(G)|$. For an integer $k\geq2$, a hypergraph $G$ is said to be $k$-uniform if every edge has size $k$.

For $u,\ v\in V(G)$, a walk from $u$ to $v$ in $G$ is defined as a sequence of vertices and edges \\ $(v_0,e_1,v_1,\cdots,v_{p-1},e_p,v_p)$, where $v_0=u$ and $v_p=v$. In this walk, each edge $e_i$ contains the vertices $v_{i-1}$ and $v_{i}$, with $v_{i-1}\neq v_{i}$ for $i=1,\cdots,p$. The value $p$ represents the length of this walk. A path is a walk where all $v_i$ and $e_i$ are distinct. A cycle is a walk containing at least two distinct edges, with all $e_i$ and $v_i$ distinct except for $v_0=v_p$. If there is a path from $u$ to $v$ for any $u,\ v\in V(G)$, then we say that $G$ is connected. A hypertree is a connected hypergraph with no cycles. 

Let $G$ be a connected hypergraph. For $u,v\in V(G)$, the distance between $u$ and $v$ is defined as the length of the shortest path from $u$ to $v$ in $G$, denoted as $d_G(u,v)$. Specifically, $d_G(u,u)=0$. The diameter of $G$ is the maximum distance between all vertex pair of $G$. The distance matrix of $G$ is the $n\times n$ matrix $D(G)=(d_G(u,v))_{u,v\in V(G)}$. The eigenvalues of $D(G)$ are called the distance eigenvalues of $G$. Since $D(G)$ is real and symmetric, the distance eigenvalues of $G$ are real. The distance spectral radius of $G$, denoted by $\rho(G)$, is the largest distance eigenvalue of $G$. Note that $D(G)$ is an irreducible nonnegative matrix. The Perron–Frobenius theorem implies that $\rho(G)$ is simple, and there is a unique positive unit eigenvector that corresponds to $\rho(G)$, which is called the distance Perron vector of $G$, denoted by $\mathbf{x}(G)$. 

The study of distance eigenvalues has a rich history in graph theory.  Its origins can be traced back to the pioneering work of Graham and Pollack(1971) \cite{graham} in which they described a relationship between the number of negative distance eigenvalues
and the addressing problem in data communication system. Edelberg et al.\cite{EDELBERG197623}and Graham and Lov\'asz\cite{GRAHAM197860} conducted early research on the distance eigenvalues, laying the foundation for subsequent studies. 
Lin and Zhou\cite{lin1} determined the unique graphs with maximum distance spectral radius among trees with given number of vertices of maximum degree.  Lin and Zhou\cite{lin2,lin3} determined the the unique $k$-uniform hypertrees with the maximum, second maximum, minimum, second minimum distance spectral radius. Moreover, they presented the exceptional structure of k-uniform unicyclic hypergraphs that have fixed sizes and possess minimum and second minimum distance spectral radii, and provided the possible structure of the $k$-uniform unicyclic hypergraph(s) of fixed size with maximum distance spectral radius. Liu et al.\cite{liu1} obtained the minimum distance spectral radius of $k$-uniform hypertrees with given matching number. Wang and Zhou\cite{MR4600138} determine the unique graph that maximizes the distance spectral radius over all cacti with fixed numbers of vertices and cycles in the context of hypertrees. For more comprehensive discussions on distance eigenvalues, interested readers may refer to the publications cited as references \cite{survey1,survey2,WangZhoou2020}.

Let $\mathbb{T}_k(m,\Delta,n)$ be the set of $k$-uniform hypertrees of size $m$ with $n$ vertices of maximum degree $\Delta$. 

In this paper, motivated by the above results, we determine the unique $k$-uniform hypertree with maximum distance spectral radius among $k$-uniform hypertrees in $\mathbb{T}_k(m,\Delta,n)$ with $\Delta \geq 3$, which can be seen as a generalization of the work in \cite{lin1}.

\section{Notations and preliminaries}
Let $G=(V,E)$ denote a (hyper)graph with vertex set $V$ and edge set $E$. For any vertex of $G$, $N_G(v)$ denotes the set of vertices adjacent to $v$, and $E_G(v)$ denotes the set of edges incident to $v$. The degree of a vertex $v$ in $G$, denoted as $d_G(v)$, is the number of edges in $E_G(v)$.
An edge $e$ in hypergraph $G$ is called a pendant edge at vertex $u$ if $u$ is the only vertex with degree greater than 1. The other vertices in $e$ are called pendant vertices of $G$ or pendant neighbors of $u$.

A $k$-uniform hypertree with $m$ edges is a hyperstar, denoted by $S^{k}_{m}$, if  all  edges share a common vertex.
A $k$-uniform loose path with $m\ge 1$ edges, denoted by $P^{k}_{m}$, is the $k$-uniform hypertree whose vertices and edges may be labeled as $(u_0, e_1, u_1,\dots, u_{m-1}, e_m, u_m)$ such that the vertices $u_1,\dots, u_{m-1}$ are of degree $2$, and all other vertices of $G$ are of degree $1$.

For a hypergraph $G$ with vertex set $V(G)$ and edge set $E(G)=\{e_1,\cdots,e_m\}$, \emph{weak vertices deletion} of $U \subset V(G)$, denoted by $G\backslash U$,  creates the hypergraph $G' = (V(G'), E(G'))$ where $V(G')=V(G)\setminus U$ and $E(G')=\{e_1',\cdots,e_m'\}$, such that $e_i'=e_i\setminus U$ for $1\leq i\leq m$. $G\backslash U$ is written as $G \backslash \{u\}$ if $U=\{u\}$. For $e_i\in E(G)$, $G''=G-e_i$ is the hypergraph with $V(G'')=V(G'')$, $E(G'')=E(G)\setminus e_i$.

For a nonempty set $X \subseteq V (G)$, $G[X]$ represents the subhypergraph induced by $X$, where $G[X]$ has vertex set $X$ and edge set of $\{e \subseteq X :  e \in E(G)\}$. Furthermore, $\sigma(X,\mathbf{x})$ denotes the sum of the entries of the vector $\mathbf{x}$ corresponding to the vertices in $X$, if $\mathbf{x}=\mathbf{x}(G)$, the alternative notation $\sigma_G(X)$ or $\sigma(X)$ can be used for $\sigma(X,\mathbf{x})$,


Given a graph $G$ and a vertex $r\in V(G)$, the \emph{rooted graph $G_r$} is the graph $G$ with the vertex $r$ labeled.  Let $V'(G)$ be the set of all nonroot vertices of $G$, that is, $V'(G)=V(G)\backslash\{r\}$.


 Let $U=(v_1,\cdots,v_s) \in V(H)^s$ and $\mathcal{G}$ be a sequence of $s$ rooted graphs $G_{r_1},G_{r_2},\cdots,G_{r_s}$ with root vertices $r_1,\dots,r_s$. The \emph{rooted product} of $H$ and $\mathcal{G}$, denoted $H(U,\mathcal{G})$, is the graph formed by attaching each root of $G_{r_i}$
in $\mathcal{G}$ to its corresponding vertex $v_i$ in $H$. We denote $G_{v_i}$ be the subhypergraph of $H(U,\mathcal{G})$ induced by $(V(G_{r_i})\cup \{v_i\})\setminus \{r_i\}$ with root vertex $v_i$


Let $H=P_m(U,\mathcal{G})$ be a rooted product of the loose path $P_m=(u_0, e_1, u_1,\dots, u_{m-1}, e_m, u_m)$ and the sequence of rooted graphs $\mathcal{G}$. For any subset $Q$ of $U$, we write $H[Q]$ as a subgraph of $H$ which is the rooted product of the subpath $P_m[Q]$ of $P_m$ and the subsequence of $\mathcal{G}$ corresponding to $Q$, where $P_m[Q]$ is the minimal subpath of $P_m$ that contains all vertices in $Q$. Especially, we denote $H^{u_i}$ as the component of $H-e_{i+1}$ containing $u_0$, $H_{u_i}$ as the component of $H-e_{i}$ containing $u_d$, respectively. We take $V'(H_{u_i})=V(H_{u_i})\setminus V(G_{u_i})$, $V'(H^{u_i})=V(H^{u_i})\setminus V(G_{u_i})$. For $0\leq c<d\leq m$, we denote $H_{u_c}^{u_d}$ as the subhypergraph of $H$ induced by $V(H)\setminus (V'(H^{u_c})\cup V'(H_{u_d})\cup V'(G_{u_c})\cup V'(G_{u_d}))$. 
 
With respect to the hypergraph mentioned above $H=P_m(U,\mathcal{G})$, if each rooted graph $G_{r_i}$ in $\mathcal{G}$ is a hyperstar, we call it a caterpillar. Furthermore, considering the situation where $U=\{u_i \mid 0\leq i \leq m\}$ and $\mathcal{G}=(G_{r_i}\mid 0\leq i \leq m)$, where the rooted graph $G_{r_i}$ is a trivial graph for $i=0, m$ or $a<i<m-b$, and it is $S_{\Delta-2}$ with its center as the root vertex for others, we can rewrite $H$ as $C(m,\Delta,a,b)$. Furthermore, $C(m,\Delta,a,b)$ is denoted as $C_k(m,\Delta,a,b)$ when $H$ is uniform with $k$. 

Let $P_{s+t}^k=(u_0,e_1,u_1,\cdots,u_{s+t-1},e_{s+t},u_{s+t})$ and $H=P_{s+t}^k(U,\mathcal{G})$, where $U=(u_1,\cdots,u_{s+t-1})$, $\mathcal{G}=(G_{r_1},\cdots,G_{r_{s+t-1}})$, and rooted graphs in $\mathcal{G}$ are $S_c^k$ with its center as the root vertex, except for $G_{r_s}$. Denote the graph $H$ by $G_c(s,t)$.


Stevanovi\'{c}  and Dragan investigated the adjacent spectral radius radius of graphs that have path-like or star-like structures attached to them \cite{MR3450917}. In what follows, we will introduce the definition of hypergraph counterparts of these graphs (refer to Fig. 1). Suppose $H$ is a connected hypergraph with vertices $u,v \in V(H)$, and $G_r$ represents a rooted graph with vertex $r$ as its root. The hypergraph, $H_{u,v}(P_s,P_t)$, is constructed from $H$ by attaching two pendant hyperpaths, $P_s$ and $P_t$, of lengths $s$ and $t$ to vertices $u$ and $v$, respectively. Furthermore, the hypergraph $H_{u,v}(P_s,P_t,G_r)$ is formed by identifying the root of a separate copy of $G_r$ at each degree 2 vertex on the pendant hyperpaths $P_s$ and $P_t$. If $u = v$, we condense the notation to $H_{u}(P_s,P_t,G_r)$. 
When $H,P_{s},P_{t},G_{r}$ are $k$ uniform hypergraphs, $H_{u,v}(P_s,P_t,G_r)$ is denoted by $H_{u,v}(s,t,G_r)$. In particular, if $G_r$ is a hyperstar with its center as the root vertex and with size $c$, we denote $H_{u}(s,t,G_r)$ and $H_{u,v}(s,t,G_r)$ as $H_{u}(s,t,c)$ and $H_{u,v}(s,t,c)$, respectively.

Let $H' = H_{u,v}(P_s,P_t,G_r)$, with $w \in V(H')$ being the pendant neighbor of a pendant vertex on subgraph $P_t$, and $w' \in V(H')$ as a pendant vertex on subgraph $P_s$. The hypergraph $H''$ emerges from $H'$ by moving the edges in $E_v(H')$, excluding the non-pendant edge of $P_t$, from $w$ to $w'$. We say that $H''$ is obtained by graft transformation. 

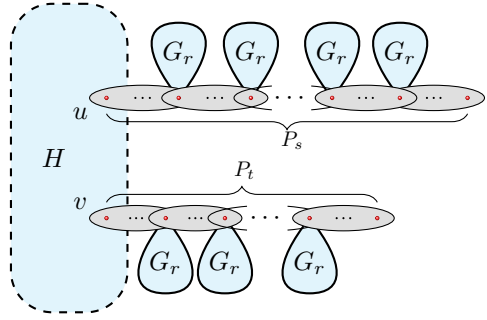
\begin{figure}[h]
    \centering

\begin{tikzpicture}[scale=0.8]
\coordinate (A1) at (-3,1) ;
\coordinate (B1) at (3,1);
\coordinate (C1) at (-0.6,1);
\coordinate (C2) at (-0.6,2);
\coordinate (E1) at (-1.8,1);
\coordinate (E2) at (-1.8,2);
\coordinate (D1) at ($(A1)!5/8!(B1)$);
\coordinate (D2) at (0.7,2);
\coordinate (F1) at (1.88,1);
\coordinate (F2) at (1.9,2);
\coordinate (A) at (-3,-1) ;
\coordinate (B) at (1.5,-1);
\coordinate (C) at ($(A)!3.5/8!(B)$);
\coordinate (C3) at (-1,-2);
\coordinate (E) at (-2.02,-1);
\coordinate (E3) at (-2,-2);
\coordinate (D) at ($(A)!6/8!(B)$);
\coordinate (D3) at (0.4,-2);

\rootG[rotate=0,scale=0.5]{(C1)};
\rootG[rotate=0,scale=0.5]{(D1)};
\rootG[rotate=0,scale=0.5]{(E1)};
\rootG[rotate=0,scale=0.5]{(F1)};
\rootG[rotate=180,scale=0.5]{(C)};
\rootG[rotate=180,scale=0.5]{(D)};
\rootG[rotate=180,scale=0.5]{(E)};

\node[draw,dashed,fill=cyan!10!white,fill opacity=0.5,thick,fit=($(A1)+(80: 15mm)$)($(A)+(-80: 1.5cm)$)($(A1)+(180: 15mm)$)($(A)+(180: 1.5cm)$),rounded corners=.55cm,inner sep=2pt]    {};

\hyperpath{A1}{C1}{2}
\partellipsis{C1}{D1}{8}{6};
\hyperpath{B1}{D1}{2}
\draw [decoration={brace,amplitude=6pt,raise=6pt,mirror}, decorate] (A1) --  (B1) node [midway,shift={(2pt,-16pt )},scale=0.8]{$P_s$};

\hyperpath{A}{C}{2}
\partellipsis{C}{D}{8}{6};
\hyperpath{B}{D}{1}
\draw [decoration={brace,amplitude=6pt,raise=6pt}, decorate] (A) --  (B) node [midway, shift={(1pt,18pt )},scale=0.8]{$P_t$};
\node at ($(A)+(150: 5mm)$){$v$};
\node at ($(A1)+(210: 5mm)$){$u$};
\node[shift={(-7mm,0pt)}] at ($(A)!0.5!(A1)$){$H$};

\end{tikzpicture}
    \caption{$H_{u,v}(P_s,P_t,G_r)$}
    \label{fig: enter-label}
\end{figure}
Lin and Zhou study the effect of some types of graft transformations to increase or decrease the distance spectral radius of connected uniform hypergraphs.\cite{lin2}
\begin{lem}\cite{lin2}\label{graft1}
    Let $H$ be a connected $k$-uniform hypergraph with $|E(H)| \geq 1$ and $u \in V(H).$ 
For integers $s\geq t\geq 1$, $\rho( H_u(s, t)) <\rho (H_u(s+1, t-1)).$ 
\end{lem}
\begin{lem}\cite{lin2}\label{graft2}
     Let $H$ be a connected $k$-uniform hypergraph with $|E(H)| \geq 2$ and $u,v\in e \in E(H)$. Suppose that $H-e$ consists of $k$ components and $d_H(u) =d_H(v) =1 $. For integers $s\geq t\geq 1$, $\rho(H_{u,v}(s, t)) <\rho(H_{u,v}(s+1, t-1)).$ 
\end{lem}

For integer $m,a,b$, let $m'=m-a-b-2$, $H=P_{m'}^k$, $u,v$ be two pendant vertices of $H$ located in different pendant edges, then $C_k(m,a,b,\Delta)$ can be rewritten as $H_{u,v}(a+1,b+1,\Delta-2)$. And the alternative notation $G_{r_s}(s,t,c)$ can be used for $G_c(s,t)$.

In Sections 3 and 4, we investigated the effects of graft transformations in two specific cases on the distance spectral radius of connected uniform hypergraphs. Based on these results, we determine the structure of the hypertree in $\mathbb{T}_k(m,\Delta,n)$ with maximum distance spectral radius.

Consider a connected $k$-uniform hypergraph $G$ and $\mathbf{x}$ be a real column vector whose coordinates are indexed by $V(G)$. For $U\subseteq V(G)$, $v\in V(G)$, we define 
 $W(U,v, \mathbf{x})=\displaystyle\sum_{\substack{u\in U_1}}x_u d(u,v)$, where $x_u$ refers to the $\mathbf{x}$ coordinate corresponding to vertex $u$. If $H$ is a subgraph of $G$, the alternative notation $W(H,v,\mathbf{x})$ can be used for $W(V(H),v,\mathbf{x})$. If $\mathbf{x}$ is the distance Perron vector of $H$, the expressions $W(U,v, \mathbf{x})$, $W(V(H),v, \mathbf{x})$, and $W(H,v,\mathbf{x})$ can be simplified to $W(U,v)$, $W(V(H),v)$, and $W(H,v)$, respectively. If $\mathbf{x}$ is an all-ones vector, the expressions $W(U,v, \mathbf{x})$, $W(V(H),v, \mathbf{x})$, and $W(H,v,\mathbf{x})$ can be rewritten as $W_0(U,v)$, $W_0(V(H),v)$, and $W_0(H,v)$, respectively. And for $U_1,U_2\subseteq V(G)$, we define $W(U_1,U_2)=\displaystyle\sum_{\substack{u\in U_1;v \in U_2}}x_u d(u,v)$, where $x_u$ refers to the $\mathbf{x}(G)$ coordinate corresponding to vertex $u$.  For $U,V_1,V_2 \subseteq V(G)$, denote $W(U,V_1)-W(U,V_2)$ as $WD(U;V_1,V_2)$.


If $G=G_{v_1} (w) G_{v_2}$, then $w$ is a cut vertex of $G$.  We have that: \\ $W(G_{v_1},u)=W(G_{v_1},w)+d(w,u)\sigma(G_{v_1})$ for any vertex of $G_{v_2}$.

Then$$\mathbf{x}^TD(G)\mathbf{x}=\sum_{\{u,v\}\subseteq V(G)}2d_G(u,v)x_ux_v=\sum_{u\in V(G)}x_u W(G,u,\mathbf{x}),$$
and $\rho$ is a distance eigenvalue with corresponding eigenvector $\mathbf{x}$ if and only if $\mathbf{x}\ne 0$ and for each $u\in V(G)$,
\begin{align}\label{eigenequ}
\rho x_u=\sum_{v\in V(G)}d_G(u,v)x_v=W(G,u,\mathbf{x}).
\end{align}
This equation is known as the eigenequation for $\rho(G)$ at $u$. By Rayleigh's principle, for a unit column vector $\mathbf{x}\in \mathbb{R}^n$, we have $\rho(G)\ge \mathbf{x}^TD(G)\mathbf{x}$, with equality if and only if $\mathbf{x}=\mathbf{x}(G)$.

We say that two vertices $u,v \in V$ in a hypergraph $G=(V,E)$ are similar if there is a graph automorphism $\phi \in Aut(G)$, s.t. $\phi(u)=v$.

\begin{lem}\label{lem1}\cite{lin2}
	Let $u,v$ be similar vertices in hypergraph $G$ and $\mathbf{x}=\mathbf{x}(G)$. Then  $x_{u}=x_{v}$.
\end{lem}

Let $u'$ be a pendant neighbor of $u$ in $k$-unifrom hypergraph $G$ with $u,u' \in e$. Define $G'=G-U$ where $U=e-\{u\}$. Consider the eigenequations for $\rho(G)$ at $u$ and $u'$.
\begin{align}
\rho(G) x_u=&W(G',u)+(k-1)x_{u'}\label{eigequ}\\
\rho(G) x_{u'}=&W(G',u')+(k-2)x_{u'}=W(G',u)+\sigma(G')+(k-2)x_{u'}\notag\\
=&W(G',u)+\sigma(G)-x_{u'}.\label{eigequ1}
\end{align}
From (\ref{eigequ}) and (\ref{eigequ1}) and Lemma \ref{lem1} we have: 
\begin{equation}\label{apuueq}
	(\rho(G)+k)x_{u'}-\rho(G) x_u=\sigma(G)
\end{equation}

Let $u$ be a vertex of degree one in $e\in G$. For any $v \in e$, we have $d_G(u,w)\geq d_G(v,w)$ for $w \in V(G)\backslash \{u\}$. 
So $(\rho(G)+1)(x_u-x_v)\geq 0$ and $x_u\geq x_v$, where the equality holds if only if $d_G(v)=1$.

Let $H$ be a connected uniform hypergraph obtained from the hypergraph $G$ by identifying two $u_1$ and
$u_2$ of $G$ with the center vertex of $k$-uniform hyperstars $S_1$ and $S_2$ with size $l\geq 0$, respectively.\\
Let $u'_1,u'_2$ be the pendant neighbor of $u_1,u_2$, respectively.
Then $\sigma(S_1)-\sigma(S_2)= x_{u_1}-x_{u_2}+l(k-1)(x_{u'_1}-x_{u'_2})$.
 From Eq. \eqref{apuueq}, we obtain: 
 \begin{align*}
 \sigma(S_1)-\sigma(S_2)=& \bigg(\frac{\rho(T)l(k-1)}{\rho(T)+k}+1\bigg )(x_{u_1}-x_{u_2}).
\end{align*}
Therefore,  
\begin{equation}\label{Ssgneq}
\sgn(\sigma(S_1)-\sigma(S_2))=\sgn(x_{u_1}-x_{u_2}).
\end{equation}

Let $G$ be a $k$-uniform hypergraph with $u, v \in V (G)$ and $e_1, \cdots, e_r \in E(G)$ such that $u \notin e_i$ and $v \in e_i$ for $1 \le i \le r$. Let $e_i'
 = (e_i \setminus \{v\})\cup \{u\}$ for $1 \le i\le r$. Let $G'$ be the hypergraph with $V (G') = V (G)$ and $E(G') = (E(G) \setminus \{e_1, \cdots, e_r\})\cup\{e'_1,\cdots, e'_r\}.$ Then we say that $G'$ is obtained from $G$ by moving edges $e_1, ..., e_r$ from $v$ to $u$.




Let $u_1,u_2 \in V(H_0)$ and $r$ be the rooted vertex of $G$. Take $H_1=H_0(u_1,G)$, $H_2=H_0(u_2,G)$. Then 
\begin{align}
\frac{1}{2}\mathbf{x}^T(D(H_2)-D(H_1))\mathbf{x}=&\sigma(V'(G),\mathbf{x})(W(H_0,u_2;\mathbf{x})-W(H_0,u_1;\mathbf{x}))\label{msgeq}
\end{align}
Take $\mathbf{x}=\mathbf{x}(H_1)$. Then 
\begin{align}
\rho(H_1)(x_{u_2}-x_{u_1})=&d(u_1,u_2)\sigma(V'(G),\mathbf{x})+(W(H_0,u_2;\mathbf{x})-W(H_0,u_1;\mathbf{x}))\label{msgeq1}
\end{align}

So $$\rho(H_2)-\rho(H_1) \geq \sigma(V'(G),\mathbf{x})(\rho(H_1)(x_{u_2}-x_{u_1}) -d(u_1,u_2)\sigma(V'(G),\mathbf{x})).$$

Let $u_1,u_2 \in V(H_0)$ and $r_1,r_2$ be the rooted vertex of $G_{r_1},\ G_{r_2}$, respectively. Take $H_1=H_0(u_1,u_2,G_{r_1},G_{r_2})$, $H_2=H_0(u_2,u_1,G_{r_2},G_{r_2})$. Then by using Eq. \eqref{msgeq} twice, we have: 
\begin{align}\label{msgeq2}
    \frac{1}{2}\mathbf{x}^T(D(H_2)-D(H_1))\mathbf{x}=&(\sigma(V'(G_{u_1}))-\sigma(V'(G_{u_2})))(W(H_0,u_2;\mathbf{x})-W(H_0,u_1;\mathbf{x}))
\end{align}
\begin{lem}\label{alem}
Suppose that $u_1,u_2$ are two distinct vertices of $H$, and $r_1,r_2$ are rooted vertices of nontrivial rooted graphs $G_{r_1}$ and $G_{r_2}$, respectively. 
Take $H_1=H((u_1,u_1),\mathcal{G}), H_2=H((u_1,u_2),\mathcal{G})$.
Then $\rho(H_2)>\rho(H_1)$ holds if one of the following two conditions is satisfied: 
\begin{enumerate}[(1).]
    \item $d_H(u_1)=d_H(u_2)=d_H(u_1,u_2)=1$;
    \item $\sigma_{H_1}(V'(G_{r_1}))+x_{u_1}\geq \sigma_{H_1}(H)$
\end{enumerate}
\end{lem}
\begin{proof}
Take $G_0=H(u_1,G_{r_1})$, then $H_1=G_0(u_1,G_{r_2})$ and $H_2=G_0(u_2,G_{r_2})$. Let $\mathbf{x}=\mathbf{x}(H_1)$.  By Eq. \eqref{msgeq}, we have
$$
\begin{aligned}
\frac{1}{2}(\rho(H_2)-\rho(H_1))\geq &\frac{1}{2}\mathbf{x}^T(D(H_2)-D(H_1))\mathbf{x}\\
=&\sigma(V'(G_{r_2}))(W(G_0,u_2)-W(G_0,u_1)) \\
=& \sigma(V'(G_{r_2}))(d_H(u_1,u_2)\sigma(V'(G_{r_1}))+W(H,u_2)-W(H,u_1))
\end{aligned}
$$
(1). If $d_H(u_1)=d_H(u_2)=d_H(u_1,u_2)=1$, then $W(H,u_1)-W(H,u_2)=0$  and $\frac{1}{2}(\rho(H_2)-\rho(H_1))>0$ follow.\\
(2). When $\sigma_{H_1}(V'(G_r))+x_u\geq \sigma_{H_1}(H)$, we have: \begin{align*}
    W(H,u_2)-W(H,u_1)&= d(u_1,u_2)x_{u_1}+W(H\setminus\{u_1\},u_2)-W(H\setminus\{u_1\},u_1)\\
    &\geq d(u_1,u_2)x_{u_1}-d(u_1,u_2)\sigma_{H_1}(H\setminus\{u_1\})>d(u_1,u_2)x_{u_1}-d(u_1,u_2)\sigma_{H_1}(H),
\end{align*}  then $\frac{1}{2}(\rho(H_2)-\rho(H_1))>0$.

\end{proof}

For $H=P_m^k(U,\mathcal{G})$ where $P_m^k=(u_0,e_1,u_1,\cdots,u_{m-1},e_m,u_m)$, $U=(u_1,\cdots,u_{m-1})$, $\mathcal{G}=(G_{r_1},\cdots,G_{r_{m-1}})$. Taking $\mathbf{x}=\mathbf{x}(H)$, by Lemma \ref{lem1}, we can use $x_{w_i}$ to donate the component of $\mathbf{x}$ that corresponds to the vertices in $e_i$ of degree $1$.

For $1 \leq i  \leq m^*-1$, we have: 
\begin{align}
	\rho(H)x_{u_i}
	  =&W(H^{u_{i-1}},u_i)+W(H_{u_{i}},u_i) +(k-2)x_{w_i}\notag\\
 \rho(H)x_{w_i}= &	W(H^{u_{i-1}},w_i)+W(H_{u_{i}},w_i)+(k-3)x_{w_{i}}\notag\\
 = &	W(H^{u_{i-1}},u_i)+W(H_{u_{i}},u_i)+\sigma(H_{u_{i}})+(k-3)x_{w_{i}}\notag\\
 =& \rho(H)x_{u_i}+\sigma(H_{u_{i}})-x_{w_{i}}
 = \rho(H)x_{u_{i-1}}+\sigma(H^{u_{i-1}})-x_{w_{i}}\label{xuw4}
\end{align}

From Eq. (\ref{xuw4}), we have: 
\begin{align}
	\rho(H)(x_{u_{i-1}}-x_{u_{i}})=&\sigma(H_{u_{i}})-\sigma(H^{u_{i-1}})\label{udefeq12}\\
	(\rho(H)+1)(x_{w_i}-x_{w_{i+1}})=& 	\sigma(H_{u_{i}}) - \sigma(H^{u_{i}})\label{wTeq2}\\
	\rho(H)(x_{u_i}+x_{u_{i-1}})=&(2\rho(H)+k)x_{w_i}-\sigma(H)\label{wTeq3}
\end{align}

From Eq. \eqref{udefeq12} and  Eq. \eqref{wTeq3}, we have: 

\begin{align}
	\rho(H)(x_{u_{i-1}}-x_{u_{i+1}})=&(2\rho(H)+k)(x_{w_{i}}-x_{w_{i+1}})\notag\\
	=& \frac{2\rho(H)+k}{\rho(H)+k-1}(\sigma(H_{u_{i+1}}) - \sigma(H^{u_{i-1}}))\label{xuiTeq}.
\end{align}
Take integers $s,t$ such that $s+t=m_0, 1\leq s-t \leq 2$.
	From \eqref{udefeq12}, we have: 
		\begin{align}
	&\rho(H)((x_{u_{t-i}}-x_{u_{s+i}})-(x_{u_{t+1-i}}-x_{u_{s+i-1}}))=\sigma(H_{u_{t+1-i}})-\sigma(H^{u_{t-i}})+\sigma(H_{u_{s+i}})-\sigma(H^{u_{s+i-1}})\notag\\
	=&2(\sigma(H_{u_{s+i}})-\sigma(H^{u_{t-i}}))+(k-2)(x_{w_{s+i}}-x_{w_{t+1-i}})\label{udeffeq1}\\
	=&2(\sigma(H_{u_s})-\sigma(H^{u_t}))+2(\sigma(H_{u_{s+i}})-\sigma(H_{u_s})+\sigma(H^{u_t})-\sigma(H^{u_{t-i}}))+(k-2)(x_{w_{s+i}}-x_{w_{t+1-i}})\notag\\
	=& \adjustbox{width=0.925\textwidth}{$2(\sigma(H_{u_s})-\sigma(H^{u_t}))+2\sum_{l=0}^{i-1}(\sigma(G_{u_{t-l}})-\sigma(G_{u_{s+l}})+(k-2)(x_{w_{t-l}}-x_{w_{s+1+l}}))+(k-2)(x_{w_{s+i}}-x_{w_{t+1-i}}).$}\label{udeffeq}
	\end{align}
	
	According to \eqref{wTeq2}, we have: 
	\begin{align}
		&(\rho(H)+1)\big((x_{w_{t-i+1}}-x_{w_{s+i}})-(x_{w_{t+2-i}}-x_{w_{s+i-1}})\big)\notag\\
  =&2(\sigma(H_{u_{s+i-1}})-\sigma(H^{u_{t-i+1}}))+\sigma(G_{t-i+1})- \sigma(G_{s+i-1})\label{wdeffeq2}\\
	=&2(\sigma(H_{u_s})-\sigma(H^{u_t}))+2(\sigma(H_{u_{s+i-1}})-\sigma(H_{u_s})+\sigma(H^{u_t})-\sigma(H^{u_{t-i+1}}))+\sigma(G_{u_{t-i+1}})- \sigma(G_{u_{s+i-1}})\notag\\
	=&2(\sigma(H_{u_s})-\sigma(H^{u_t}))+2\sum_{l=0}^{i-2}(\sigma(G_{u_{t-l}})-\sigma(G_{u_{s+l}})+(k-2)(x_{w_{t-l}}-x_{w_{s+1+l}}))+\sigma(G_{u_{t-i+1}})- \sigma(G_{u_{s+i-1}}).\label{wdeffeq}
	\end{align}

\begin{lem}\label{nlem1} Suppose $1\leq s-t \leq 2$ and $\sgn(\sigma(G_{u_{t-l}})-\sigma(G_{u_{s+l}}))=\sgn(x_{u_{t-l}}-x_{u_{s+l}})$ for $0\leq l \leq r$. Then \\
(1). $x_{u_{t-l}}-x_{u_{s+l}}$ and $ \sigma(H_{u_s})-\sigma(H^{u_t})$ have common sign for $0\leq l \leq r+1$. \\
(2). $x_{w_{t+1}}-x_{w_{s}}$ and $\sigma(H_{u_s})-\sigma(H^{u_t})$ have common sign or $ x_{w_{t+1}}-x_{w_{s}}=0$. $x_{w_{t+1-l}}-x_{w_{s+l}}$ and $\sigma(H_{u_s})-\sigma(H^{u_t})$ have common sign for $1\leq l \leq r+1$.
\end{lem}

\begin{proof}
From Eq. \eqref{udefeq12} and \eqref{xuiTeq}, we get (1) is true for $l=0$ and $x_{w_{t+1}}-x_{w_{s}}$ and $\sigma(H_{u_s})-\sigma(H^{u_t})$ have common sign or $ x_{w_{t+1}}-x_{w_{s}}=0$.

Suppose that for $0\leq l\leq m-1\leq r$, $x_{u_{t-l}}-x_{u_{s+l}}$ and $\sigma(H_{u_s})-\sigma(H^{u_t})$ have common sign, $x_{w_{t+1}}-x_{w_{s}}$ and $\sigma(H_{u_s})-\sigma(H^{u_t})$ have common sign or $ x_{w_{t+1}}-x_{w_{s}}=0$, and for $1\leq l\leq m-1\leq  r$, $x_{w_{t+1-l}}-x_{w_{s+l}}$ and $\sigma(H_{u_s})-\sigma(H^{u_t})$ have common sign , 
from Eq. \eqref{udeffeq} and \eqref{wdeffeq}, we have $x_{u_{t-m}}-x_{u_{s+m}}$ , $x_{w_{t+1-m}}-x_{w_{s+m}}$ and $\sigma(H_{u_s})-\sigma(H^{u_t})$ have common sign.
By induction hypothesis, we complete the proof.
\end{proof}
From Eq. \eqref{udeffeq} and \eqref{wdeffeq} and Lemma \ref{nlem1}, we have: 
\begin{cor}\label{ncor1}
Suppose $1\leq s-t \leq 2$, $\sigma(H_{u_s})-\sigma(H^{u_t})> 0$ and   $\sgn(\sigma(G_{u_{t-l}})-\sigma(G_{u_{s+l}}))=\sgn(x_{u_{t-l}}-x_{u_{s+l}})$ for $0\leq l \leq r$. Then we have : \\
$x_{u_{t-l}}-x_{u_{s+l}}> x_{u_{t-l+1}}-x_{u_{s+l-1}}> 0$ and 
$x_{w_{t-l+1}}-x_{w_{s+l}}> x_{w_{t-l+2}}-x_{w_{s+l-1}}\ge 0$ for $0\leq l \leq r+1$. 
\end{cor}

\begin{cor}\label{ncor2}
Suppose $1\leq s-t \leq 2$, $\sigma(H^{u_t})-\sigma(H_{u_s})> 0$ and $\sgn(\sigma(G_{u_{t-l}})-\sigma(G_{u_{s+l}}))=\sgn(x_{u_{t-l}}-x_{u_{s+l}})$ for $0\leq l \leq r$. Then we have : \\
$x_{u_{t-l}}-x_{u_{s+l}}< x_{u_{t-l+1}}-x_{u_{s+l-1}}< 0$ and 
$x_{w_{t-l+1}}-x_{w_{s+l}}< x_{w_{t-l+2}}-x_{w_{s+l-1}}\leq 0$ for $0\leq l \leq r+1$.
\end{cor}

\begin{lem}\label{lem6}
Let $H=G_c(s,t)$. For $s\ge t\ge 2$, we have $\sigma(V'(H^{u_s}))+x_{u_s}> \sigma(V'(H_{u_{s+1}}))+x_{u_{s+1}}$, $\sigma(V'(H^{u_s}))\geq \sigma(V'(H_{u_s}))$.\\
\end{lem} 
\begin{proof}
By Lemma \ref{nlem1}, we know that $\sigma(H^{u_{s-1}})-\sigma(H_{u_{s+1}})$ and $x_{u_{s+i}}-x_{u_{s-i}},x_{w_{s+i}}-x_{w_{s+1-i}}$ have common sign for $1\le i\le t$.
Note that$$\sigma(H^{u_{s-1}})-\sigma(H_{u_{s+1}})\ge -(\sum_{i=1}^{t}(\sigma(G_{u_{s+i}})-\sigma(G_{u_{s-i}}))+(k-2)\sum_{i=2}^{t}(x_{w_{s+i}}-x_{w_{s+1-i}}))$$
By \eqref{Ssgneq},
$\sigma(H_{s+i})-\sigma(H_{s-i})$ and $x_{u_{s+i}}-x_{u_{s-i}}$ have common sign, which implies $\sigma(H^{u_{s-1}})\geq\sigma(H_{u_{s+1}})$, then  $\sigma(V'(H^{u_s}))+x_{u_s}> \sigma(H^{u_{s-1}})\geq \sigma(H_{u_{s+1}})> \sigma(V'(H_{u_{s+1}}))+x_{u_{s+1}}$.

By $x_{u_{s+1}}-x_{u_{s}}=\frac{2\rho(G)+k}{\rho(G)}(x_{w_{s+1}}-x_{w_{s}})\geq 0$. Note that: 
$$\rho(H)(x_{u_{s+1}}-x_{u_{s-1}})=2(\sigma(V'(H^{u_s}))-\sigma(V'(H_{u_s})))+(k-2)(x_{w_{s+1}}-x_{w_{s}}),$$
and $\rho(G)>k$, we have $\sigma(V'(H^{u_s}))\geq \sigma(V'(H_{u_s}))$.
\end{proof}

\section{Distance spectral radius of caterpillars in $\mathbb{T}_k(m,\Delta,n)$}

In this section, we investigate the distance spectral radius of the caterpillars in $T=C_k(m^*,\Delta,a,b)$, and we give the structure of the caterpillars in $\mathbb{T}_k(m,\Delta,n)$ with the maximum distance spectral radius. Taking $\mathbf{x}=\mathbf{x}(T)$, by Lemma \ref{lem1}, we can use $x_{v_i}$ to donate the component
of $\mathbf{x}$ that corresponds to the pendant neighbor of $u_i$, and $x_{w_i}$ to donate the component of $\mathbf{x}$ that corresponds to the vertices of $e_i$ of degree $1$.Denote $G_{u_i}$ as $S_i$.

These are some important facts about the components of the vector $\mathbf{x}=\mathbf{x}(T)$.

\begin{fact}\label{1f}

Let integers $a,b,m^*$ as defined above. Then we have
$$
0<x_{u_{0}}-x_{u_{m^*}}<x_{u_{1}}-x_{u_{m^*-1}}<\cdots< x_{u_{a}}-x_{u_{m^*-a}}< x_{u_{a+1}}-x_{u_{m^*-a-1}}
$$
and 
$$
0<x_{w_{1}}-x_{w_{m^*}}<x_{w_{2}}-x_{w_{m^*-1}}<\cdots< x_{w_{a+1}}-x_{w_{m^*-a}}<x_{u_{a+1}}-x_{u_{m^*-a-1}}.
$$
\end{fact}
\begin{proof}

Take integers $p+q=m^*$, $1\leq q-p\leq 2$. From $0\leq a \leq b-2$, we know that $S_{p-j}\cong S_{q+j}$ or $S_{q+j}\cong S^k_{\Delta-2}$ and $S_{p-j}$ is trivial. For $0\le j\le p$,  if $x_{u_{p-j}}\leq x_{u_{q+j}}$,  then $\sigma(S_{p-j})\le \sigma(S_{q+j})$ by \eqref{Ssgneq}. 

 Suppose $\sigma(T_{u_q})-\sigma(T^{u_p})\leq 0$, 
 by Lemma \ref{nlem1}, we have $x_{u_{p-i}}\le x_{u_{q+i}}, x_{w_{p+1-i}}\le x_{w_{q+i}}$ for all $0\le i \le p$. Since $b\ge a+2$, there exists an $0\le i_0\le p$ such that $\sigma(S_{i_0})<\sigma(S_{m^*-1-i_0})$. Then we have $\sigma(T_{u_{q}})>\sigma(T^{u_p})$, a contradiction. Therefore,  $$\sigma(T_{u_q})-\sigma(T^{u_p})> 0.$$
From Eq.  \eqref{udefeq12} and \eqref{xuiTeq} and $1\leq q-p \leq 2$, we have $x_{u_{p}}> x_{u_{q}}$ and $x_{w_{p+1}} \ge x_{w_{q}}$.
  
Suppose $x_{u_0}\le x_{u_{m^*}}$. Then we can show $x_{w_1} \leq x_{w_{m^*}}$ due to $x_{u_0}=x_{w_1}$ and $x_{u_{m^*}}=x_{w_{m^*}}$. 
 
For $p+q=m^*$, by Eq. \eqref{udeffeq1} and \eqref{wdeffeq2} we have: \\
If $x_{u_j}\leq x_{u_{m^{*}-j}} \text{ and }x_{w_{j}}\le x_{w_{m*-j+1}}$ for $0\leq j\leq i$, then \\
$x_{u_{i+1}}-x_{u_{m^{*}-i-1}}\leq x_{u_{i}}-x_{u_{m^*-i}} \leq 0$, $x_{w_{i+1}}-x_{w_{m^*-i}}\leq x_{w_{i}}-x_{w_{m^*-i+1}}\leq 0$.\\ So $x_{u_{i}}\le x_{u_{m^{*}-i}}$ for $0\leq i \leq p$. It follows that $x_{u_p}\le x_{u_{q}}$, a contradiction. \\ Thus $x_{u_0}>x_{u_{m^*}}$ and $x_{w_1}>x_{w_{m^*}}$.

From $S_i\cong  S_{m^*-i}$ for $1\leq i\leq a$, we have $\sgn(\sigma(S_i)-\sigma(S_{m^*-i}))=x_{u_i}-x_{u_{m^*-i}}$ \text{for}\ $1\leq i\leq a$. 
By Eq. \eqref{udeffeq1} and \eqref{wdeffeq2}, take $s=l+1, t=m^*-l-1$, we have for $0\le l\le a$, $x_{u_{l+1}}-x_{u_{m^*-l-1}}>x_{u_{l}}-x_{u_{m^*-l}}>0$  and for $ 1\leq k\leq a$, $x_{w_{k+1}}-x_{w_{m^*-k}}>x_{w_{k}}-x_{w_{m^*-k+1}}>0$ by induction on $l$.

Note that
$\sigma(T_{u_{a+1}})=\sigma(T)-\sigma(V'(T^{u_{a+1}}))$, $\sigma(T^{u_{m^*-a-1}})=\sigma(T)-\sigma(V'(T_{m^*-a-1}))$. By\\ $\sgn(\sigma(S_i)-\sigma(S_{m^*-i}))=x_{u_i}-x_{u_{m^*-i}}\  \text{for}\ 1\leq i\leq a,$\

we have $\sigma(V'(T^{u_{a+1}}))>\sigma(V'(T_{m^*-a-1}))$, then we have $$\sigma(T_{u_{a+1}})<\sigma(T^{u_{m^*-a-1}}).$$

From  Eq. \eqref{xuw4}, we have $$(\rho(T)+1)x_{w_{a+1}}-\rho(T)x_{u_{a+1}}=\sigma(T_{u_{a+1}})<\sigma(T^{u_{m^*-a-1}})=(\rho(T)+1)x_{w_{m^*-a}}-\rho(T)x_{u_{m^*-a-1}}.$$

So $\rho(T)(x_{u_{a+1}}-x_{u_{m^*-a-1}})>(\rho(T)+1)(x_{w_{a+1}}-x_{w_{m^*-a}})$.

Thus $ x_{u_{a+1}}-x_{u_{m^*-a-1}}>x_{w_{a+1}}-x_{w_{m^*-a}}>0$.
\end{proof}

\begin{fact}\label{2f} Let $m'=m^*+a-b+1$, $p'+q'=m'$ and $1 \leq q'-p'\leq 2$. Then we have: 
$$
0<x_{u_{p'}}-x_{u_{q'}}<x_{u_{p'-1}}-x_{u_{q'+1}}<\cdots< x_{u_{a+1}}-x_{u_{m^*-b}}
$$
and 
$$
0<x_{w_{p'+1}}-x_{w_{q'}}<x_{w_{p'}}-x_{w_{q'+1}}<\cdots< x_{w_{a+2}}-x_{w_{m^*-b}}<x_{u_{a+1}}-x_{u_{m^*-b}}.
$$
\end{fact}
\begin{proof}
From $a+2\leq b, 1\leq q-p \leq 2$ and $p+q=m^*$,  we have $p'\leq p$ and $q' \leq q$. 

By Fact 1, we have $\sigma(T^{u_{p'}})\leq \sigma(T^{u_{p}})<\sigma(T_{u_{q}}) \leq \sigma(T_{u_{q'}})$.

From  Lemma \ref{nlem1} and Corollary \ref{ncor1}, we have 
$x_{u_i}-x_{u_{m'-i}}>x_{u_{i+1}}-x_{u_{m'-i-1}}>0,$\\$ x_{w_{i+1}}-x_{w_{m'-i}}>x_{w_{i+2}}-x_{w_{m'-i-1}}>0$ for $a+1\le i\le p'-1$.

Then we have $\sigma(V'(T_{u_{a+1}}^{u_{p'}}))+x_{u_{p'}}>\sigma(V'(T_{q'}^{u_{m^*-b}}))+x_{u_{q'}}$.

Note that $$\sigma(T^{u_{a+1}})+\sigma(V'(T_{u_{a+1}}^{u_{p'}}))+x_{u_{p'}}=\sigma(T^{p'})<\sigma(T_{u_{q'}})=\sigma(T_{u_{m^*-b}})+\sigma(V'(T_{q'}^{u_{m^*-b}}))+x_{u_{q'}}.$$
we have $$\sigma(T^{u_{a+1}})<\sigma(T_{u_{m^*-b}}).$$

From Eq. \eqref{xuw4}, we have $$(\rho(T)+1)x_{w_{a+2}}-\rho(T)x_{u_{a+1}}=\sigma(T^{u_{a+1}})<\sigma(T_{u_{m^*-b}})=(\rho(T)+1)x_{w_{m^*-b}}-\rho(T)x_{u_{m^*-b}}.$$

So $\rho(T)(x_{u_{a+1}}-x_{u_{m^*-b}})>(\rho(T)+1)(x_{w_{a+2}}-x_{w_{m^*-b}})$.

Thus $x_{u_{a+1}}-x_{u_{m^*-b}}>x_{w_{a+2}}-x_{w_{m^*-b}}>0$.
\end{proof}
\begin{fact}\label{3f}Let $m''=2m^*-a-b-1$, $p''+q''=m''$ and $1\leq q''-p''\le 2$. Then we have
$$ x_{u_{m^*-b}}-x_{u_{m^*-a-1}}<x_{u_{m^*-b+1}}-x_{u_{m^*-a-2}}<\cdots< x_{u_{p''}}-x_{u_{q''}}<0$$
and $$ x_{w_{m^*-b+1}}-x_{w_{m^*-a-1}}<x_{w_{m^*-b+2}}-x_{u_{m^*-a-2}}<\cdots< x_{u_{p''+1}}-x_{u_{q''}} \leq 0$$
\end{fact}

\begin{proof}

For $p''+q''=m''$ and $1\leq q''-p''\leq 2$, 
by Lemma \ref{nlem1}, we have
$x_{u_{p''-i}}-x_{u_{q''+i}}$ having the same sign as $\sigma(T^{u_{p''}})-\sigma(T_{u_{q''}})$ for $0 \leq i \leq p''+b-m^{*}$ and $x_{w_{p''-i+1}}-x_{w_{q''+i}}$ having the same sign as $\sigma(T^{u_{p''}})-\sigma(T_{u_{q''}})$ or $x_{w_{p''-i+1}}-x_{w_{q''+i}}=0$ for $0 \leq i \leq p''+b-m^{*}$.\\
From Eq. \eqref{Ssgneq}, we have $\sigma(T^{u_{m*-b-1}})-\sigma(T^{u_{p''}})-(\sigma(T_{u_{m^*-a}})-\sigma(T_{u_{q''}}))= c(\sigma(T^{u_{p''}})-\sigma(T_{u_{q''}}))$ for some positive number c.
By Fact 1, we have $\sigma(T^{u_{a}})>\sigma(T_{u_{m^*-a}})$. 
Since $a + b<m^*-2$, we have
$$\begin{aligned}
&\sigma(T^{u_{p''}})-\sigma(T_{u_{q''}})
    = \sigma(T^{u_{p''}})-\sigma(T^{u_{a}})+\sigma(T^{u_{a}})-\sigma(T_{u_{q''}})\\
&> \sigma(T^{u_{p''}})-\sigma(T^{u_{m*-b-1}})+\sigma(T^{u_{a}})-\sigma(T_{u_{m^*-a}})-(\sigma(T_{u_{q''}})-\sigma(T_{u_{m^*-a}}))\\
    &>-(\sigma(T^{u_{m^*-b-1}})-\sigma(T^{u_{p''}})-(\sigma(T_{u_{m^*-a}}))-\sigma(T_{u_{q''}}))=-c(\sigma(T^{u_{p''}})-\sigma(T_{u_{q''}})).
\end{aligned}$$
Thus, $\sigma(T^{u_{p''}})-\sigma(T_{u_{q''}})>0$. By Corollary \ref{ncor2}, we have $x_{u_{p''-i}}-x_{u_{q''+i}}<x_{u_{p''-i+1}}-x_{u_{q''+i-1}}<0$ and $x_{w_{p''-i+1}}-x_{w_{q''+i}}< x_{w_{p''-i+2}}-x_{w_{q''+i-1}}\leq 0$  for $0 \leq i \leq p''+b-m^{*}$.
\end{proof}

\begin{lem}\label{lem5}
Let $k$-uniform hypertree $T=C_k(m^*,\Delta,a,b)$ with  $m^*>a+b$. For $ a+2 \leq  b$, we have $\rho(T)<\rho(C_k(m^*,\Delta,a+1,b-1)).$
\end{lem}
\begin{proof}
Let $T'$ be the graph obtained from $T$ by moving all pendant edges at ${u_{m*-b}}$ from  to ${u_{a+1}}$.  Then $T'\cong C_k(m^*,\Delta,a+1,b-1)$. 

Let $r=d(u_{m^*-b},u_{a+1})$ and take $\mathbf{x}=\mathbf{x}(T)$, $H=T\setminus \{V'(S_{m^*-b})\}$\\

By Eq. \eqref{msgeq}, we have the following.\\
 $$\frac{1}{2}(\rho(T')-\rho(T))= \sigma(V'(S_{m^*-b}))(W(H,u_{a+1})-W(H,u_{m^*-b})).$$

To show $\rho(T')>\rho(T)$, we only need to show $J=W(H,u_{a+1})-W(H,u_{m^*-b})>0$. 

Let $c'(r)=W_0(V'(T_{u_{a+1}}^{m^*-b}),u_{p'+1})+r|V(T^{u_{a+1}})|$. By Fact 3, $x_{u_{a+1}}>x_{u_{m^*-b}}$.

To show $J>0$, we only need to show the followings: 

(1). $\rho(T)> c'(r)$;

(2). $2J>(\rho(T)- c'(r))(x_{u_{a+1}}-x_{u_{m^*-b}})$.

For (1), since $\rho(T)$ is bounded below the minimum sum of the row of $D(T)$, it suffices to prove that the sum of each row of $D(T)$ is greater than $c'(r)$.

For $0\le j\le a$, $$\begin{aligned}
W_0(T,u_j)&\ge W_0(T_{u_{a+1}},u_a)+W_0(T^{u_a},u_j)+W_0(V'(T_{u_a}^{u_{a+1}}),u_a)\\
&>W_0(T_{m^*-a-1},u_a)+W_0(T_{u_a}^{u_{m^*-b}},u_a)>c'(r).
\end{aligned}$$
Then, 
$W_0(T,w_j)=W_0(T,u_j)+|V(T_{u_j})|-1>c'(r)$.

For pendant vertices $v_j$ of $u_j$, $W_0(T,v_j)-W_0(T,u_j)=|V(T)|-k>0$. 

For $m^*-b\le j\le m^*, u_j,v_j,w_j$ it can be proved by a similar argument.

For $a+1\le j\le m^*-b-1$, we have $W_0(V'(T^{m^*-b}_{u_{a+1}}),u_j)\geq W_0(V'(T_{u_{a+1}}^{m^*-b}),u_{p'+1})$, then$$\begin{aligned}
W_0(T,u_j)&=W_0(T^{u_{a+1}},u_j)+W_0({T_{u_{m^*-b}}},u_j)+W_0(V'(T^{m^*-b}_{u_{a+1}}),u_j)\\
&>W_0(T^{u_{a+1}},u_j)+W_0({T_{u_{m^*-a-1}}},u_j)+W_0(V'(T^{m^*-b}_{u_{a+1}}),u_j) > c'(r)
\end{aligned}$$
Then, $W_0(T,w_j)=W_0(T,u_j)+|V(T_{u_j})|-1>c'(r)$.
Thus, we have $ \rho(T)>c'(r)$.

For (2), by Eq. \eqref{msgeq1} and Fact 3, we have:  
    
       $$  \rho(T)(x_{u_{a+1}}-x_{u_{m^*-b}})=J+r\sigma(V'(S_{m^*-b}))< 2J+r\sigma(V'(S_{m^*-a-1}))-J.$$

we only need to show that: $J>-c'(r)(x_{u_{a+1}}-x_{m^*-b})+r\sigma(V'(S_{m^*-a-1})).$\\



From $V(H)=V(T^{u_{a+1}})\sqcup V(T_{u_{m^*-a-1}})\sqcup V'(T_{u_{m^*-a-1}},u_{m^*-b})\sqcup V'(T_{m^*-b}^{m^*-a-1})\sqcup \{u_{m^*-b}\}$, we know that
$J=J_1+J_2+J_3+rx_{u_{m^*-b}}$, where 
$$\begin{aligned}
J_1=&W(T^{u_{a+1}}\cup T_{u_{m^*-a-1}},u_{a+1})-W(T^{u_{a+1}}\cup T_{u_{m^*-a-1}},u_{m^*-b}),\\
J_2=&W(V'(T_{m^*-b}^{m^*-a-1}),u_{a+1})-W(V'(T_{m^*-b}^{m^*-a-1}),u_{m^*-b}), \\
J_3=&W(V'(T_{u_{a+1}}^{m^*-b}),u_{a+1})-W(V'(T_{u_{a+1}}^{m^*-b}),u_{m^*-b}). 
\end{aligned}$$
Since $\forall w\in V'(T_{m^*-b}^{m^*-a-1})$, $d_T({u_{a+1}},w)-d_T(u_{m^*-b},w)>0$, we have $J_2>0.$

By Fact 1 and Fact 3,$x_{u_{a+1}}-x_{u_{m^*-b}}>x_{u_{a+1}}-x_{u_{m^*-a-1}}>0$,\\
note that $-r|V(T^{u_{a+1}})|=W_0(T^{u_{a+1}},u_{a+1})-W_0(T^{u_{a+1}},u_{m^*-b})<0$,  we have 
 $$
\begin{aligned}
&J_1=W(T^{u_{a+1}}\cup T_{u_{m^*-a-1}},u_{a+1})-W(T^{u_{a+1}}\cup T_{u_{m^*-a-1}},u_{m^*-b})\\
    >&(W_0(T^{u_{a+1}},u_{a+1})-W_0(T^{u_{a+1}},u_{m^*-b}))(x_{u_{a+1}}-x_{u_{m^*-a-1}})+r\sigma(V'(S_{m^*-a-1})\\
    >&-r|V(T^{u_{a+1}})|(x_{u_{a+1}}-x_{u_{m^*-b}})+r\sigma(V'(S_{m^*-a-1}),
\end{aligned}$$

Note that $-(W_0(V'(T_{u_{a+1}}^{m^*-b}),u_{p'+1}))<W_0(V'(T_{u_{a+1}}^{p'+1}),u_{a+1})-W_0(V'(T_{u_{a+1}}^{p'+1}),u_{m^*-b})<0$, by Fact 2, $x_{u_{a+1}}-x_{u_{m^*-b}} >0$, we have : 
$$\begin{aligned}
    &J_3=W(V'(T_{u_{a+1}}^{m^*-b}),u_{a+1})-W(V'(T_{u_{a+1}}^{m^*-b}),u_{m^*-b})\\
    &>(W_0(V'(T_{u_{a+1}}^{p'+1}),u_{a+1})-W_0(V'(T_{u_{a+1}}^{p'+1}),u_{m^*-b}))(x_{u_{a+1}}-x_{u_{m^*-b}})\\
    &>-(W_0(V'(T_{u_{a+1}}^{m^*-b}),u_{p'+1}))(x_{u_{a+1}}-x_{u_{m^*-b}})
    \end{aligned}$$
Then $$\begin{aligned}
    J&=J_1+J_2+J_3+rx_{m^*-b}>-(r|V(T^{u_{a+1}})|+W_0(V'(T_{u_{a+1}}^{m^*-b}),u_{p'+1}))(x_{u_{a+1}}-x_{u_{m^*-b}})\\
    &+r\sigma(V'(S_{m^*-a-1})+rx_{m^*-b}\\
    &>-c'(r)(x_{u_{a+1}}-x_{m^*-b})+r\sigma(V'(S_{m^*-a-1})).
\end{aligned}$$



\end{proof}
\begin{lem}\label{nlem3}
  Let $T$ be the $k$-uniform hypertree in $\mathbb{T}_k(m,\Delta,n)$ with maximum distance spectral radius, where $\Delta\ge 3$ and $n \ge 2$. Then each edge $e$ of $T$ has at most $2$ vertices in $e$ with degree at least $2$.
\end{lem}
\begin{proof}
 If there is an edge $e$ of $T$, such that there are at least $3$ vertices of degree at least $2$ in $e$. 
 Let $P=(u_0,e_1,u_1,\cdots,u_{d-1},e_d,u_d)$  be one of the longest paths in $T$ which contains $e$.  Then there exists some vertex $w$ in $e$ such that $E'=E_T(w)-E(P)\neq \emptyset$. Suppose $u_i,u_{i+1}\in e$, $d_P(u_i)=d_P(u_{i+1})=2$. Let $T_1$ be the component of $T-e$ containing $u_0$, $T_2$ be the component of $T-e$ containing $u_d$, $H_{w}$ be the connected components of $T-e$ containing $w$, $T_3$ be the component of $T\setminus (\{u_i\}\cup \{u_{i+1}\})$ containing $w$. Suppose without loss of generality that $\sigma(T_1)\geq \sigma(T_2)$, then let $T'$ be the graph obtained from $T$  moving edges in $E'$ from $w$ to a pendent neighbor $u_d$. 
Then $T'\in \mathbb{T}_k(m,\Delta,n)$. Note that:  \begin{equation}\label{eq16}
	\begin{aligned}
	\frac{1}{2}&(\rho(T')-\rho(T))\ge \frac{1}{2}\mathbf{x}^T(D(T')-D(T))\mathbf{x}\\
	&>(d_T(u_{i+1},u_d))(\sigma(T_1)-\sigma(T_2))\sigma(V'(H_w))+((d_T(u_{i+1},u_d)+1)x_{w}+x_{u_{d-1}}-x_{u_d})\sigma(V'(H_w)).
\end{aligned}
\end{equation}
And $$\rho(T)(x_{u_d}-x_{u_{d-1}})=\sigma(T_1)+\sigma(T_2)-\sigma(e_d)+\sigma(T_3)+x_{u_{d-1}}-x_{u_{d}}$$
$$\rho(T)x_{w_{i_0}}>\sigma(T_{1})+\sigma(T_{2})+\sigma(T_3)-x_{w_{i_0j}}.$$Thus $$\rho(T)((d_T(u_{i+1},u_d)+1)x_{w_{i_0}}+x_{u_{d-1}}-x_{u_d})> \sigma(e_d)+x_{u_d}-x_{u_{d-1}}>0,$$
we get $\rho(T')>\rho(T)$ from \eqref{eq16}, a contradiction. Thus each edge $e$ of $T$ has at most $2$ vertices in $e$ with degree at least $2$.
\end{proof}
By a similar argument we can obtain the following conclusion: 
\begin{cor}\label{ncor3}
    Let $T$ be a caterpillar in $\mathbb{T}_k(m,\Delta,n)$, where $\Delta\ge 3$ and $n\geq 2$. Then each edge $e$ of $T$ has at most $2$ vertices in $e$ with degree at least $2$.
\end{cor}
Now, we are ready to identify the structure of the caterpillar in $\mathbb{T}_k(m,\Delta,n)$ with maximum distance spectral radius.  In graph theory, a branch at a vertex $u$ in a tree is a maximal subtree containing $u$ as a pendent vertex.
\begin{thm}\label{thm1}
Let $T$ be a caterpillar in $\mathbb{T}_k(m,\Delta,n)$ with $\Delta\ge 3$ and $n \ge 2$. Let $m^*=m-n(\Delta-2)$, then $\rho(T)\le \rho(C_k(m^*,\Delta,\lfloor \frac{n}{2}\rfloor,\lceil \frac{n}{2}\rceil))$, where equality holds if and only if $T\cong C_k(m^*,\Delta,\lfloor \frac{n}{2}\rfloor,\lceil \frac{n}{2}\rceil)$.
\end{thm}   
\begin{proof}
Assume $T$ to be a caterpillar with the maximum distance spectral radius in $\mathbb{T}_k(m,\Delta,n)$, and let $\mathbf{x}=\mathbf{x}(T)$. 
Suppose that $P=(u_0,e_1,u_1,\cdots,u_{d-1},e_d,u_d)$ is a diameter path of $T$. By  Corollary \ref{ncor3}, for $w\in P\setminus (\displaystyle\cup_{i=1}^{d-1}\{u_i\})$, $d_T(w)=1$.

Let $W$ be the set of vertices of degree $\Delta$ in $T$. Suppose that $d_T(u_{i})\ge 3$ for some $u_{i}\in V(T)\setminus W$.  Then there exists a branch $T_i$ of $T$ at $u_{i}$ containing an edge $e \in E_{T}(v_i)\backslash E(P)$. Suppose without loss of generality that $\sigma(T^{u_{i}})\ge \sigma(T_{u_{i}})$.  Let $T'$ be the graph obtained from $T$ by moving edge $e_{i}$ from $u_{i}$ to $u_d$, then $T'$ is a caterpillar in $\mathbb{T}_k(m,\Delta,n)$. By Lemma \ref{alem} $\rho(T')>\rho(T)$, a contradiction. Thus, the degree of any vertex of $V(T)\setminus W$ is at most two.

If $H$ has no vertex of degree $2$, then we have $T\cong C_k(m^*,\Delta,\lfloor \frac{n}{2}\rfloor,\lceil \frac{n}{2}\rceil)$. Otherwise, suppose that $d_T(u_1)=d_T(u_{d-1})=2$ and $d_T(u_j)=\Delta$ for some $j$ with $2\le j\le d-2$.   
Suppose without loss of generality that $\sigma(T^{u_j})\ge \sigma(T_{u_j})$. Let $T''$ be the graph obtained from $T$ by moving the pendant edges at $u_j$ from $u_j$ to $u_{d-1}$, then $T''$ is a caterpillar in $\mathbb{T}_k(m,\Delta,n)$. By Lemma \ref{alem}, $\rho(T)<\rho(T'')$, a contradiction. Thus, $\max(d_T(u_{2}),d_T(u_{d-1}))=\Delta$. Without loss of generality, suppose $d_T(u_{d-1})=\Delta$.

Assume that $d_{u_1}=2$ and let $U$ be the sets of vertices of degree $2$ in $T$. If $T\setminus U$ has exactly one component with size greater than $1$, then $T\cong C_k(m,{\Delta},0,n)$. By Lemma \ref{lem5}, we have $\rho(T)=\rho(C_k(m^*,{\Delta},0,n))<\rho(C_k(m^*,\Delta,\lfloor \frac{n}{2}\rfloor,\lceil \frac{n}{2}\rceil))$, a contradiction. Therefore, $T\setminus U$ has at least two non-trivial components and there are vertices $u_i$ and $u_j$ with $2\le i<j\le d-2$ such that $d_T(u_i)=\Delta$ and $d_T(u_j)=2$. 
If $\sigma(T^{u_i})\ge \sigma(T_{u_i})$, let $T_1'$ be the graph obtained from $T$ by moving the pendant edges at $u_i$ from $u_i$ to $u_{j}$. By Lemma \ref{alem} we have $\rho(T_1')>\rho(T)$, a contradiction. If $\sigma(T^{u_i})< \sigma(T_{u_i})$, let $T_2'$ be the graph obtained from $T$ by moving the pendant edges at $u_i$ from $u_i$ to $u_{1}$, by Lemma \ref{alem} we have $\rho(T_2')>\rho(T)$, a contradiction. Then $d_{u_1}=\Delta$, and $T\setminus U$ has at least two components with size greater than $1$.

Suppose $T\setminus U$ has at least $3$ components with size greater than $1$. Then there are three vertices $u_i,u_j,u_l$ with $2\le i<j<l\le d-2$ in $T$ such that $d_T(u_j)=\Delta, d_T(i)=d_T(l)=2$.   
Suppose without loss of generality that $\sigma(T^{u_j})\ge\sigma(T_{u_j})$. Let $T'''$ be the graph obtained from $T$ by moving the pendant edges at $u_j$ from $u_j$ to $u_{l}$, then $T'''$ is a caterpillar in $\mathbb{T}_k(m,\Delta,n)$. By Lemma \ref{alem} $\rho(T''')>\rho(T)$, a contradiction.Thus $T\setminus U$ has exactly $2$  components with size greater than $1$, then $T\cong C_k(m^*,\Delta,a,b)$, for some $a+b=n, a\ge1,b\ge 1$. By Lemma \ref{lem5}, $T\cong C_k(m^*,\Delta,\lfloor \frac{n}{2}\rfloor,\lceil \frac{n}{2}\rceil).$
\end{proof}

\section{Distance spectral radius of trees in $\mathbb{T}_k(m,\Delta,n)$}
In the following section, we turn our attention to finding the extreme graph within the class $\mathbb{T}_k(m,\Delta,n)$ that maximizes the distance spectral radius. We begin by considering the graph $H=G_c(s,t)$. Set $\mathbf{x}=\mathbf{x}(H)$, following the notation used in Lemma \ref{lem1}, we denote $x_{v_i}$ for the component of $\mathbf{x}$ pertaining to the pendant neighbor of $u_i$ for $i\ne s$, set one of these vertices as $v_i$, and let $x_{w_i}$ be the component of $\mathbf{x}$ that corresponds to the vertices in $e_i$ of degree 1.

For $i=1,\cdots,c$, let $h_i, f_i$ be two vertices of a pendant edge $e_{s,i}$ of $k$-uniform hypergraph $G_i$ with $d_{G_i}(h_i)=1$ and $d_{G_i}(f_i)>1$ if $|E(G_i)|>1$. $F_i=\{f_i\}$ for $|H_i|=1$, and $F_i$ is the non-trivial connected component of $H_i-e_{s,i}$ for $|E(H_i)|>1$.
The rooted graph $G_{r_s}$ of $G_c(s,t)$ is the graph obtained from $G_1,\ldots,G_c$ identifying $h_1,\ldots,h_c$ as a new vertex $r_s$, which is labeled as the root of $G_{r_s}$. For the graph $H=G_c(s,t)$, set $\mathbf{x}=\mathbf{x}(H)$, let $q_i$ be one of the vertices of degree one in $e_{s,i}\setminus({h_i}\cup {f_i})$. By Lemma \ref{lem1}, we denote $x_{q_i}$ the component of $\mathbf{x}$ that corresponds to any vertex of degree one in $e_{s,i}\setminus(\{h_i\}\cup \{f_i\})$. We have following lemma holds: 
\begin{lem}\label{lem7}Setting $|E(G_1)|>1$, for $s\ge t\ge 2$,  we have $$\rho(G_c(s+1,t-1))>\rho(G_c(s,t))$$ holds if one of the following conditions satisfied:  (i) $c=1$; (ii) $c\ge 2$, $G_1$ is a loose path and $|E(G_i)|=1$ for $2\le i\le c$; (iii) $c\ge 2$ and there is a vertex of $V(G_{1})$ with at least $c$ adjacent pendant edges in $G_c(s,t)$.
\end{lem}
\begin{proof}
Take $H=G_c(s,t)$, $\mathbf{x}=\mathbf{x}(H)$. For $(i)$, Let $H'$ be the graph obtained from $H$ by moving $E_{G_{u_s}}(u_s)$ from $u_s$ to $u_{s+1}$, moving $E_{G_{u_{s+1}}}(u_{s+1})$ from $u_{s+1}$ to $u_{s}$, then  $H'\cong G_1(s+1,t-1)$.
Using Eq. \eqref{msgeq2}, we have
\begin{align*}
\frac{1}{2}(\rho(H')-\rho(H))&\ge \frac{1}{2}\mathbf{x}^T(D(H')-D(H))\mathbf{x}\\
&=(\sigma(V'(G_{u_{s}}))-\sigma(V'(G_{u_{s+1}}))(\sigma(V'(H^{u_{s}}))+x_{u_s}-\sigma(V'(H_{u_{s+1}}))-x_{u_{s+1}}).
\end{align*}
By Lemma \ref{lem6}, we have $\sigma(V'(H^{u_{s}}))+x_{u_s}>\sigma(V'(H_{u_{s+1}}))+x_{u_{s+1}}$. Hence, to prove $\rho(G')>\rho(G)$, it suffices to show that $\sigma(V'(G_{u_{s}}))-\sigma(V'(G_{u_{s+1}}))>0$.

Choose $z\in V(G_{u_s})$ such that $d_H(z,f_1)=\max_{v\in V(F_1)}d_H(v,f_1)=d$, then $z$ is in a pendant edge $e'$, let the set of pendant vertices in $e'$ be $Z$. From the eigenequation of $G$ we have$$\begin{aligned}
&\rho(G)x_{f_1}=d(k-1)x_z+3(k-1)x_{v_{s+1}}+x_{u_s}+W(V'(G_{u_s}\setminus Z),f_1)+W(H\setminus (V(G_{u_s})\cup V'(G_{u_{s+1}})),f_1),\\
&\rho(G)x_{q_1}\ge (d+1)(k-1)x_z+3(k-1)x_{v_{s+1}}-x_{q_1}+x_{u_s}\\
&\ \ +W(V'(G_{u_s}\setminus Z),f_1)+W(G\setminus (V(G_{u_s})\cup V'(G_{u_{s+1}})),q_1),\\
&\rho(G)x_z=(k-2)x_z+(d+3)(k-1)x_{v_{s+1}}+(d+1)x_{u_s}+W(V'(G_{u_s}\setminus Z),z)+W(H\setminus (V(G_{u_s})\cup V'(G_{u_{s+1}})),z),\\
	&\rho(G)x_{v_{s+1}}\le (d+3)(k-1)x_z+(k-2)x_{v_{s+1}}+2x_{u_s}+W(V'(G_{u_s}\setminus Z),f_1)\\
 &\ \ +3\sigma(V'(G_{u_s}\setminus Z))+W(G\setminus (V(G_{u_s})\cup V'(G_{u_{s+1}})),v_{s+1})
\end{aligned}$$

Note that for $w \in V(H)\setminus(V(G_{u_s})\cup V'(G_{u_{s+1}}))$,

$$d_H(f_1,w)+(k-2)d_H(q_1,w)+(k-1)d_H(z,w)-(k-1)d_H(v_{s+1},w)\ge 0,$$ we have $$\begin{aligned}
&\rho(G)(x_{f_1}+(k-2)x_{q_1}+(k-1)x_z-(k-1)x_{v_{s+1}})\\
\ge &(1-k^2)x_z+(d+3)(k-1)^2x_{v_{s+1}}\\
- &(k-2)x_{q_1}+(k-1)(W(V'(G_{u_s}\setminus Z),z)-3\sigma(V'(G_{u_s}\setminus Z))).
\end{aligned}$$
Then by $\sigma(V'(G_{u_s}))=\sigma(V'(G_{u_s}\setminus Z))+(k-1)x_z$ and $\sigma(V'(G_{u_s}))>x_{q_1}$, we have: $$\begin{aligned}
	&(\rho(G)+4(k-1))(\sigma(V'(G_{s}))-\sigma(V'(G_{s+1})))\\
 &\ge \rho(G)(x_{f_1}+(k-2)x_{q_1}+(k-1)(x_z-x_{v_{s+1}}))
	+4(k-1)(\sigma(V'(G_{s}))-\sigma(V'(G_{s+1})))\\
        &>\rho(G)(x_{f_1}+(k-2)x_{q_1}+(k-1)(x_z-x_{v_{s+1}}))+3(k-1)(\sigma(V'(G_{s}))+(k-1)x_{q_1}-4(k-1)\sigma(V'(G_{s+1}))\\
	&> (d-1)(k-1)^2x_{v_{s+1}}+(k-1)(W(V'(G_{u_s}\setminus Z),z))>0.
\end{aligned}$$
Thus $\sigma(V'(G_{s}))-\sigma(V'(G_{s+1}))>0$.

For $(ii)$, let $e$ be the adjacent edge of $f_1$ different from $e_{1}$, $H''$ be the graph obtained from $H$ by moving edge $e$ from $f_1$ to $v_{s+1}$, then $H''\cong G_c(s+1,t-1)$. Then \begin{equation}\label{eq14}
\begin{aligned}
	\frac{1}{2}(\rho(H'')-\rho(H))&\ge \frac{1}{2}\mathbf{x}^T(D(H'')-D(H))\mathbf{x}\\
        &=(\sigma(F_1)-x_{f_1})(\sigma(V'(H^{u_s}))+x_{u_s}+3x_{f_1}+2(k-1)x_{q_1}+(c-1)(k-1)x_{q_2})\\
        &-(\sigma(F_1)-x_{f_1})(\sigma(V'(H_{u_s}))-(k-2)x_{w_{s+1}}+kx_{v_{s+1}})\\
	&> (\sigma(F_1)-x_{f_1})(x_{u_s}+3x_{f_1}+2(k-2)x_{q_1}+(k-1)x_{q_2}-kx_{v_{s+1}})
\end{aligned}
\end{equation}
For any $w\in V(G)\setminus (V(e_{s,1})\cup V(e_{s,2})\cup V(e_{s+1,1}))$, we have $$
d_H(w,u_s)+3d_H(w,f_1)+2(k-2)d_H(q_1,w)+(k-1)d_H(q_2,w)-kd_H(v_{s+1},w)\ge 0$$
Thus $$\begin{aligned}
&\rho(G)(x_{u_s}+3x_{f_1}+2(k-2)x_{q_1}+(k-1)x_{q_2}-kx_{v_{s+1}})\\
&\ge (k-2)x_{u_s}+(k-5)x_{f_1}+(k-2)(k-4)x_{q_1}+(2k-3)(k-1)x_{q_2}+(8k^2-11k+4)x_{v_{s+1}},
\end{aligned}$$
and then $$\begin{aligned}
&(\rho(G)+1)(x_{u_s}+3x_{f_1}+2(k-2)x_{q_1}+(k-1)x_{q_2}-kx_{v_{s+1}})\\
&\ge (k-1)x_{u_s}+(k-2)x_{f_1}+(k-2)^2x_{q_1}+2(k-1)^2x_{q_2}+4(2k-1)(k-1)x_{v_{s+1}}>0.
\end{aligned}$$
Therefore, $x_{u_s}+3x_{f_1}+2(k-2)x_{q_1}+(k-1)x_{q_2}-kx_{v_{s+1}}>0$, from \eqref{eq14}, we have $\rho(G'')>\rho(G)$.

For $(iii)$, let $y\in G_1$ such that $Y$ is a hyperstar $S^k_c$ which edges in $E(Y)$ are pendant edge of $y$. Denote $V'(Y)=V(Y)-\{y\}$, by Lemma \ref{lem1}, we denote $x_{y'}$ for the component of $\mathbf{x}$ pertaining to the vertices in $V'(Y)$.
Let $H'''$ be the graph obtained from $H$ by moving $E_{G_{u_s}}(u_s)$ from $u_s$ to $u_{s+1}$, moving $E_{G_{u_{s+1}}}(u_{s+1})$ from $u_{s+1}$ to $u_{s}$, then $H'''\cong G_c(s+1,t-1)$. Then by Eq. \eqref{msgeq2} \begin{equation}\label{eq15}
\begin{aligned}\frac{1}{2}(\rho(H''')-\rho(H))&\ge \frac{1}{2}\mathbf{x}^T(D(H''')-D(H))\mathbf{x}\\
	&=(\sigma(V'(G_{u_s}))-\sigma(V'(G_{u_{s+1}})))(\sigma(V'(H^{u_{s}}))+x_{u_s}-\sigma(V'(H_{u_{s+1}}))-x_{u_{s+1}}).
\end{aligned}
\end{equation}
By Lemma \ref{lem6}, we have $\sigma(V'(H^{u_{s}}))+x_{u_s}> \sigma(V'(H_{u_{s+1}}))+x_{u_{s+1}}$. To show $\rho(G''')>\rho(G)$, we only need to show that $$\sigma(V'(G_{u_s}))\ge \sigma(V'(G_{u_{s+1}})).$$
Take 
$$
\begin{aligned}
    V_1&=V'(H^{u_s}),V_2=V'(H_{u_{s+1}}),V_3= V(e_{s+1}),V_4= V(\cup_{j=2}^cF_j),\\
    V_5&= V(F_1\setminus V'(Y)),V_6= N_{G_{u_s}}(u_s)\setminus \cup_{j=1}^c \{f_i\},V_7= V'(Y),V_8= V'(G_{u_{s+1}})
\end{aligned}
$$

Let $I=N_{G_{u_{s}}}(u_s)\cup V'(Y)$. Note that for $v\in V_5\cup V_7$, $d_H(v,q_i)\geq d_H(v,f_i)$ and $d_H(y,f_1)\geq 0$, we have: 
$$\begin{aligned}
    WD(V_1;I,V'(G_{u_{s+1}}))&=c(k-1)W(V'(H^{u_s}),u_s)+(d_H(y,f_1)+1)\sigma(V'(H^{u_s}))>0,\\
    WD(V_2;I,V'(G_{u_{s+1}}))&=c(k-1)W(V'(H_{u_{s+1}}),u_{s+1})+(d_H(y,f_1)+4)\sigma(V'(H_{u_{s+1}}))>0,\\
    WD(V_3;I,V'(G_{u_{s+1}}))&=c(k-1)\big ((d_H(y,f_1)+4)x_{u_{s+1}}
    +(k-2)(d_H(y,f_1)+3)x_{w_{s+1}}+(1+d_H(y,f_1))x_{u_s}\big )>0,\\
    WD(V_4;I,V'(G_{u_{s+1}}))&> WD(V_4;V'(Y),V'(G_{u_{s+1}}))\geq 0,\\
    WD(V_5;I,V'(G_{u_{s+1}}))&\geq \sum_{i=1}^c(k-1)W(V_5,x_{f_i})+ WD(V_5;V'(Y),V'(G_{u_{s+1}}))\\
    &= (k-1)\sum_{i=1}^c\big(W(F_1\setminus  V'(Y),y)
    +(d_H(f_1,f_i)-2)\sigma(F_1\setminus  V'(Y))\big)>-2c(k-1)\sigma(F_1\setminus  V'(Y)),\\
    WD(V_6;I,V'(G_{u_{s+1}}))&\geq \sum_{i=1}^c(k-2)(ck-k-c)x_{q_i}\geq 0,\\
    WD(V_7;I,V'(G_{u_{s+1}}))&\geq \sum_{i=1}^c(k-1)W(V_7,x_{f_i})+ WD(V_7;V'(Y),V'(G_{u_{s+1}}))\\
    &\geq (k-1)c\big((k-1)c+(2-3k)\big)x_{y'},\\
    WD(V_8;I,V'(G_{u_{s+1}}))&\geq (5ck-5c+k) \sigma(V'(G_{u_{s+1}}))>5c(k-1)\sigma(V'(G_{u_{s+1}})).
\end{aligned}$$

By $I\subseteq V'(G_{u_s})$, and $V(G)=\bigsqcup_{i=1}^{8} V_i$,
we have : \begin{equation}\label{neq2}
    \begin{aligned}
&\rho(G)(\sigma(I)-\sigma(V'(G_{u_{s+1}})))=\sum_{i=1}^8 WD(V_i;I,V'(G_{u_{s+1}}))\\
&>c(k-1)\big(-2\sigma(F_1\setminus  V'(Y))+((c-3)k-c+2)x_{y'}+5\sigma(V'(G_{u_{s+1}}))\big).
\end{aligned}\end{equation}
From $\sigma(V'(G_{u_s}))-\sigma(F_1\setminus  V'(Y))> \sigma( V'(Y))$, we have \begin{equation*}
    \begin{aligned}
    &(\rho(G)+2c(k-1))(\sigma(V'(G_{u_s}))-V'(G_{u_{s+1}}))\\
  &  > c(k-1)\big(2\sigma( V'(Y))+   ((k-1)c+(2-3k))x_{y'}+3\sigma(V'(G_{u_{s+1}}))\big)\\
&>c(k-1)\big((3(k-1)(c-1)-1)x_{y'}+3\sigma(V'(G_{u_{s+1}}))\big)\\
&>3c(k-1)\sigma(V'(G_{u_{s+1}}))>0
\end{aligned}\end{equation*}

which implies that $\sigma(V'(G_{u_s}))>\sigma(V'(G_{u_{s+1}}))$. Thus, $\rho(G''')>\rho(G)$.
\end{proof}
Now, we are ready to identify the structure of the hypertree in $\mathbb{T}_k(m,\Delta,n)$ with maximum distance spectral radius.
\begin{thm}\label{thm2}
Let $T$ be the $k$-uniform hypertree in $\mathbb{T}_k(m,\Delta,n)$ with maximum distance spectral radius, where $\Delta\ge 3$ and $n \ge 2$, then $T\cong C_k(m^*,\Delta,\lfloor \frac{n}{2}\rfloor,\lceil \frac{n}{2}\rceil)$, where $m^*=m-n(\Delta-2)$.
\end{thm}
\begin{proof}
Let $\mathbf{x}=\mathbf{x}(T)$. By Lemma \ref{nlem3},each edge $e$ of $T$ has at most $2$ vertices in $e$ with degree at least $2$, then by  an argument analogous to the one given in the  proof of Theorem \ref{thm1}, the degree of any vertex of degree less than $\Delta$ is at most $2$.

Let $U$ be the set of vertices of degree $2$ in $T$ and $|U|=k_2$. For any $u\in V(T)$ with degree $\Delta$, let $E_T(u)=\{e_1,e_2,\cdots,e_{\Delta}\}$ and $T_i$ be the branch of $T$ at $u$ containing $e_i$, for $1\le i\le \Delta$. Now, we prove that $U\subseteq V(T_i)$ for some $i$ with $1\le i\le \Delta$. Suppose for the sake of contradiction that  $w_1,w_2\in U,\ w_1\in V(T_1)$, and $w_2\in V(T_2)$. Without loss of generality, assume that $\sigma(T_1)\ge \sigma(T_2)$. Let $T''$ be the graph obtained from $T$ by moving the edges $e_3,\cdots e_{\Delta}$ from $u$ to $w_2$. Then $T''\in \mathbb{T}_k(m,\Delta,n)$, by Lemma \ref{alem}, $\rho(T'')>\rho(T)$, a contradiction. Then we have $U\subseteq V(T_i)$ for some $i$ with $1\le i\le \Delta$.

Suppose that $T$ is not a caterpillar, we have $T\cong G_{\Delta-2}(s,t)$ for some nontrivial hypergraph $G_{r_s}$, and some integers $s$ and $t$ with $s\ge t\ge 2$. Obviously, $G_{\Delta-2}(s+1,t-1)\in \mathbb{T}_k(m,\Delta,n)$. If $\Delta=3$, then by Lemma \ref{lem7} $(i)$, $\rho(G_{\Delta-2})(s+1,t-1)>\rho(T)$, a contradiction.

Suppose that $\Delta\ge 4$. If $m=n(\Delta-1)+1$, then $U=\emptyset$, implying that $V'(G_{r_s})$ contains a vertex $v$ with degree $\Delta$, then $v$ has at least $\Delta-2$ adjacent pendant edges in $T$. By Lemma \ref{lem7} $(iii)$, $\rho(G_{\Delta-2}(s+1,t-1))>\rho(T)$, a contradiction. Suppose that $m>n(\Delta-1)+1$, then $k_2\ge 1$,  if $V'(G_{r_s})$ contains a vertex of degree $\Delta$, then there is a vertex in $V'(G_{r_s})$ with at least $\Delta-2$ adjacent pendant edges in $T$, according to Lemma \ref{lem7} $(iii)$, $\rho(G_{\Delta-2}(s+1,t-1))>\rho(T)$, a contradiction. Therefore, $V'(G_{r_s})$ contains no vertex of degree $\Delta$ in $T$, that is, each vertex of $V(G_{r_s})$ is of degree $1$ or $2$. Since $T$ is not a caterpillar, we have $G_{i_0}$ is a loose path for some $1\leq i_0\leq c$ and $E(G_i)=1$ for $i\neq i_0$, according to Lemma \ref{lem7} $(ii)$,  $\rho(G_{\Delta-2}(s+1,t-1))>\rho(T)$, a contradiction. Thus $T$ is a caterpillar in $\mathbb{T}_k(m,\Delta,n)$, then by Theorem \ref{thm1}, $T\cong C_k(m^*,\Delta,\lfloor \frac{n}{2}\rfloor,\lceil \frac{n}{2}\rceil).$
\end{proof}



We have shown that if $T$ is a $k$-uniform hypertree in $\mathbb{T}_k(m,\Delta,n)$ with maximum distance spectral radius, then it must be isomorphic to $C_k(m^*,\Delta,\lfloor \frac{n}{2} \rfloor, \lceil \frac{n}{2} \rceil)$ where $m^*=m-n(\Delta-2)$. Furthermore, if $T$ is a $k$-uniform hypertree with $m$ edges and $n$ vertices of maximum degree having maximum distance spectral radius, where $n\geq 2$. Take $H_1=C_k(m_1,\Delta,\lfloor \frac{n}{2} \rfloor, \lceil \frac{n}{2} \rceil)$, $H_2=C_k(m_2,\Delta-1,\lfloor \frac{n}{2} \rfloor, \lceil \frac{n}{2} \rceil)$, where $|E(H_1)|=|E(H_2)|=m$, $\Delta\leq \frac{m-1}{n}$, then by similar argument as in the proof of Theorem \ref{thm2} and using Lemma \ref{alem}, we have: 

$$\rho (H_1)<\rho (H_2), $$ we can obtain the following corollary:

\begin{cor}
	If $T$ is a $k$-uniform hypertree with $m$ edges and $n$ vertices of maximum degree having maximum distance spectral radius, where $2\le n\le \lfloor \frac{m-1}{2}\rfloor$. Then $\rho(T)\le \rho (C_k(m^*,3,\lfloor \frac{n}{2} \rfloor, \lceil \frac{n}{2} \rceil))$, with equality if and only if $T\cong C_k(m^*,3,\lfloor \frac{n}{2} \rfloor, \lceil \frac{n}{2} \rceil)$,  where $m^*=m-n$.
\end{cor}

We study the effects of the graft transformations in two specific cases on the distance spectral radius of connected uniform hypergraphs. Based on Lemma \ref{lem7}, we can make following conjecture.
\begin{con}
    $|E(G_{r_s})|>c$ and $d_{G_{r_s}}(r_s)\geq c$, then for $s\ge t\ge 2$, $$\rho(G_c(s+1,t-1))>\rho(G_c(s,t)).$$ 
\end{con}

About the graft transformation at different vertices, Lemma \ref{lem5} presents a special case, based on which we can propose the following question.
\begin{question}Given uniform hypergraphs $H$ and $G_r$ and the integers with $p \geq q \geq 1$, what conditions ensure that $\rho(H_{u,v}(p,q;G_r))>\rho(H_{u,v}(p+1,q-1;G_r))$?
\end{question}

Let $P_1=(u_0,e_1,u_1,e_2,u_2)$ with $|e_1|=9$, $|e_2|=3$, $H=P_2=(u,e_3,v)$ with $e_3=3$, $d_H(u)=d_H(v)=1$, $P_3=(u_3,e_4,u_4,e_5,u_5)$ with $|e_4|=3$, $|e_5|=6$. Let $H'$ be the hypergraph obtained from $H$ by attaching $u_2$, $u_3$ at $u$, $v$, respectively, and $H''$ be the hypergraph obtained form $H'$ by moving $e_1$ from $u_1$ to $u_5$. By calculation, we have $\rho(H)\approx 53.04$, $\rho(H')\approx 46.91$, then $\rho(H)>\rho (H')$, which shows that Lemma \ref{graft2} does not hold in the case of non-uniform hypergraphs.


For the $3$-uniform hypergraph $H=C_3(5,3,1,2)\in \mathbb{T}_3(8,3,3)$, let $H'$ be the non-uniform hypertree obtained from $H$ by moving a vertex of degree $1$ in $e_2$ from $e_2$ to $e_1$. By calculation, we have  $\rho(H)\approx 45.33$, $\rho(H')\approx 46.31$, then $\rho(H')>\rho(H)$. Hence, $T=C_k(m^*,\Delta,\lfloor \frac{n}{2}\rfloor,\lceil \frac{n}{2}\rceil)$ is not the hypertree with the maximal distance spectral radius among hypertrees with
fixed order $|V(T)|$, size $|E(T)|$, maximal degree $\Delta$, and number of vertices of maximum degree.  It is nature to ask the following question.
\begin{question}
Which hypertrees have the the maximum distance spectral radius among all
hypertrees with fixed order, size, maximal degree, and number of vertices of maximum degree?
\end{question}
\newrefcontext[sorting=none]
\printbibliography

\end{document}